\documentclass[10pt,a4paper]{article}

\usepackage{amsfonts}
\usepackage{amssymb}
\usepackage{amsmath}
\usepackage{amsthm}
\usepackage[english]{babel}
\usepackage{enumerate}
\usepackage{graphicx}

\newcommand{\NN}{\mathbb{N}}
\newcommand{\RR}{\mathbb{R}}

\newcommand{\ints}{\int\limits}

\newcommand{\ep}{\varepsilon}

\newcommand{\al}{\alpha}
\newcommand{\si}{\sigma}

\newcommand{\x}{\hat{x}}
\newcommand{\ie}{i.e.}

\newcommand{\à}{\`a}

\newcommand{\Om}{\Omega}
\newcommand{\pa}{\partial}

\newtheorem{theorem}{Theorem}[section]

\newtheorem{lemma}[theorem]{Lemma}
\newtheorem{proposition}[theorem]{Proposition}

\newtheorem{example}[theorem]{Example}

\theoremstyle{remark}
\newtheorem{rmk}[theorem]{Remark}

\DeclareMathOperator{\tr}{\mathrm{tr}}
\DeclareMathOperator{\diver}{\mathrm{div}}

\DeclareMathOperator{\dist}{\mathrm{dist}}

\renewenvironment{proof}{
  \noindent{\it Proof.}\ }{\hspace*{\fill}
  \begin{math}\Box\end{math}\medskip}

\title{A viscosity equation for minimizers of a class of very degenerate elliptic functionals}

\author{Giulio Ciraolo\thanks{Dipartimento di Matematica e
Informatica, Universit\`a di Palermo, Via Archirafi 34, 90123, Italy, ({\tt g.ciraolo@math.unipa.it}).}
}

\begin{document}

\maketitle

\begin{abstract}
We consider the functional
\begin{equation*}
J(v) = \int_\Omega [f(|\nabla v|) -  v] dx,
\end{equation*}
where $\Omega$ is a bounded domain and $f:[0,+\infty)\to \RR$ is a convex function vanishing for $s\in [0,\sigma]$, with $\sigma>0$. We prove that a minimizer $u$ of $J$ satisfies an equation of the form
\begin{equation*}
\min(F(\nabla u, D^2 u), |\nabla u|-\sigma)=0
\end{equation*}
in the viscosity sense.
\end{abstract}

\section{Introduction} \label{section introd}
Let $\Omega$ be a bounded domain in $\RR^N,\: N\geq 2$, with boundary $\pa \Omega$ of class $C^{2,\al}$, with $0<\al<1$. We consider the variational problem
\begin{equation}\label{J(v)}
\inf \{J(v):\ v \in W_0^{1,\infty}(\Omega)\},\quad \textmd{where }\ \ J(v) = \ints_\Omega [f(|\nabla v|) -  v] dx;
\end{equation}
here, the function $f: [0,+\infty) \to \RR$ is convex, monotone, nondecreasing and we assume that there exists $\si>0$ such that
\begin{subequations} \label{f hp}
\begin{eqnarray}
&& f\in C^1([0,+\infty)) \cap C^3((\si,+\infty)); \label{f0} \\
&& f(0)=0 \textmd{ and } \lim\limits_{s\to +\infty} \dfrac{f(s)}{s} = +\infty; \label{f1} \\
&& f'(s)=0\textmd{ for every } 0 \leq s \leq \si ;  \label{f2} \\
&& f''(s)>0 \textmd{ for } s>\si . \label{f3}
\end{eqnarray}
\end{subequations}
Functionals of this kind occur in the study of complex-valued solutions of the {\it eikonal} equation (see \cite{MT1}--\cite{MT4}), as well as in the study of problems linked to traffic congestion (see \cite{BCS}) and in variational problems which are relaxations of non-convex ones (see \cite{CarMul}). We have in mind the following two main examples of a function $f$:
\begin{equation} \label{f ikonal}
f(s)=\begin{cases} 0, & 0\leq s \leq 1 \\
\frac{1}{2}[s\sqrt{s^2-1} - \log(s+\sqrt{s^2-1})], & s>1,
\end{cases}
\end{equation}
which arises from the study of complex-valued solutions of the eikonal equation, and
\begin{equation} \label{f congestion traffic}
f(s)=\begin{cases}
\frac{1}{q}(s-1)^q, & s >1,\\
0, & 0 \leq s \leq 1,
\end{cases}
\end{equation}
$q>1$, which is linked to traffic congestion problems.

Since $f$ vanishes in the interval $[0,\si]$, problem \eqref{J(v)} is strongly degenerate and, as far as we know, few studies have been done. Besides the papers cited before, we mention \cite{Br} and \cite{SV} where regularity issues were tackled.

In this paper, we shall prove that the minimizer $u$ of \eqref{J(v)} satisfies an equation of the form
\begin{equation}\label{eq min introd}
\min \Big(F(\nabla u,D^2 u), |\nabla u|-\si \Big) = 0
\end{equation}
in the viscosity sense (see Theorems \ref{viscosityEqForV f C2} and \ref{viscosityEqForV} for the meaning of $F$).

Our strategy is to approximate $J$ by a sequence of less degenerating functionals so that the minimizers of the corresponding variational problems converge uniformly to $u$; this is done in Section \ref{section 2}. Then, the machinery of viscosity equations applies and, in Section \ref{section 3}, we prove that $u$ satisfies \eqref{eq min introd}. To prove Theorems \ref{viscosityEqForV f C2} and \ref{viscosityEqForV}, which are our main results, we make use of techniques which have been used in the context of the $\infty$-Laplace operator (see for instance \cite{BK},\cite{Je},\cite{JLM}).

\section{Preliminary results} \label{section 2}
We start by recalling some well-known facts. Since $\Omega$ is bounded and $\pa \Omega$ is of class $C^{2,\al}$ then the following \emph{uniform exterior sphere condition} holds: there exists $\rho>0$ such that for every $x_0\in\pa \Omega$ there exists a ball $B_{\rho}(y)$ of
radius $\rho$ centered at $y=y(x_0) \in \RR^N \setminus \overline{\Omega}$ such that $\overline{B_{\rho}(y)} \cap \overline{\Omega} = \overline{B_{\rho}(y)} \cap \pa \Omega$ and $x_0 \in \partial{B_{\rho}(y)}$.

Notice that, since $f$ satisfies \eqref{f0}-\eqref{f3}, the functional $J$ is
differentiable and a critical point $u$ of $J$ satisfies the
problem
\begin{equation} \label{Euler eq 1}
\begin{cases}
- \diver \left( \dfrac{f'(|\nabla u|)}{|\nabla u|} \nabla u \right) = 1, & \textmd{in } \Omega,\\
u=0, &  \textmd{on } \pa \Omega,
\end{cases}
\end{equation}
in the weak sense, i.e.
\begin{equation}\label{Euler eq 2}
\ints_\Omega \dfrac{f'(|\nabla u|)}{|\nabla u|} \nabla u \cdot \nabla \phi dx = \ints_\Omega \phi dx, \quad \textmd{for every } \phi\in C_0^1(\Omega).
\end{equation}

It will be useful in the sequel to have at hand the solution of \eqref{Euler eq 1} when
$\Om$ is the ball of given radius $R$ (centered at the origin): it is given by
\begin{equation} \label{eq soluzione palla}
u_R(x) = \ints_{|x|}^R g'\Big( \frac{s}{N} \Big) ds ,
\end{equation}
where
\begin{equation*}
g(t)=\sup\{st-f(s):\ s\geq 0\}
\end{equation*}
is the Fenchel conjugate of $f$ (see for instance \cite{Cr} and \cite{FGK}).

It is clear that, when $\si=0,$ \eqref{J(v)} has a unique solution, since $f$ is strictly convex.
When $\si>0,$ the uniqueness of a minimizer for \eqref{J(v)} is proved in \cite{CMS}.

In this section we shall approximate the functional $J$ by a sequence of strictly convex functionals
\begin{equation}\label{Jn}
J_n(v) =  \ints_\Omega [f_n(|\nabla v|) -  v] dx,
\end{equation}
$n\in\NN$, which are less degenerating than $J$ (see Proposition \ref{prop:approx minimizers} for the assumptions on the functions $f_n$) and prove some uniform bounds for the minimizers $u_n$ of
\begin{equation}\label{Jn(v) pb min}
\inf \{J_n(v):\ v \in W_0^{1,\infty}(\Omega)\}.
\end{equation}

Notice that, if $f_n \in C^1([0,+\infty)) \cap C^3((0,+\infty))$ satisfies \eqref{f1} and it is such that $f_n'(0)=0$ and $f_n''(s)>0$ for $s>0$, then the minimizer $u_n$ of \eqref{Jn} is unique and satisfies
\begin{equation}\label{Euler eq 2 fn}
\ints_\Omega \dfrac{f_n'(|\nabla u_n|)}{|\nabla u_n|} \nabla u_n \cdot \nabla \phi dx = \ints_\Omega \phi dx, \quad \textmd{for every } \phi\in C_0^1(\Omega).
\end{equation}

We shall say that $w\in W^{1,\infty}(\Omega)$ is a subsolution of \eqref{Euler eq 2 fn} if
\begin{equation*}
\ints_\Omega \dfrac{f_n'(|\nabla w|)}{|\nabla w|} \nabla w \cdot \nabla \phi dx \leq \ints_\Omega \phi dx, \quad \textmd{for every } \phi\in C_0^1(\Omega) \ \textmd{ with } \phi\geq 0,
\end{equation*}
and that $w\in W^{1,\infty}(\Omega)$ is a supersolution of \eqref{Euler eq 2 fn} if
\begin{equation*}
\ints_\Omega \dfrac{f_n'(|\nabla w|)}{|\nabla w|} \nabla w \cdot \nabla \phi dx \geq \ints_\Omega \phi dx, \quad \textmd{for every } \phi\in C_0^1(\Omega) \ \textmd{ with } \phi\geq 0.
\end{equation*}

Let $u_n$ and $v_n$ be a subsolution and a supersolutions of \eqref{Euler eq 2 fn}, respectively. Then, the following \emph{weak comparison principle} holds: if $u_n \leq v_n$ on $\pa \Omega$ then $u_n \leq v_n$ in $\overline{\Omega}$  (see Lemma 3.7 in \cite{FGK}).

It will be useful to define the following $P$-function (see \cite{FGK}):
\begin{equation} \label{P}
P_n (x) = \Phi_n(|\nabla u_n(x)|) + \frac{2}{N} u_n(x), \quad x\in \overline{\Omega},
\end{equation}
where
\begin{equation} \label{Phi}
\Phi_n(t) = 2 \ints_0^t s f_n''(s) ds.
\end{equation}
To avoid heavy notations, in Lemmas \ref{lemma:gradients equibounded} and \ref{lemma: bounds II} we drop the dependence on $n$.

\begin{lemma} \label{lemma:gradients equibounded}
Let $f \in C^1([0,+\infty)) \cap C^3((0,+\infty))$ be such that $f'(0)=0$ and $f''(s)>0$ for $s>0$ and let $u$ be the solution of \eqref{J(v)}. Then, $|\nabla u|$ attains its maximum on the boundary of $\Omega$ and the following estimate holds:
\begin{equation} \label{bound Du}
|\nabla u(x)| \leq M , \quad x\in\overline{\Omega},
\end{equation}
with
\begin{equation}\label{M}
M=g'\left(\frac{\rho}{N-1} \Big(e^{\frac{(N-1)R^*}{\rho}} -1 \Big) \right),
\end{equation}
where $g$ is the Fenchel conjugate of $f$, $R^*=\sup\{|x-y|:\ x,y \in \pa \Omega\}$ and $\rho$ is the radius of the uniform exterior sphere.

Furthermore,
\begin{equation}\label{bounds u}
0 \leq u(x) \leq \min \left( \ints_{0}^{R^*} g'\Big( \frac{s}{N} \Big) ds,\ \frac{N}{2} \Phi(M) \right) \quad x \in \overline{\Omega}.
\end{equation}
\end{lemma}
\begin{proof}
Since $u$ is a minimizer of $J$, it is easy to show that $u\geq 0$. Being $R^*$ the diameter of $\Omega$, there exist a ball of radius $R^*$ that contains $\Omega$ (we can assume that such ball is centered at the origin). Since $u_{R^*}(x) \geq 0$ for $ x \in \pa \Omega$, the weak comparison principle implies that
\begin{equation}
u(x) \leq u_{R^*}(x) \ \textmd{ for every } x \in \overline{\Omega}.
\end{equation}
From $u_{R^*}(x) \leq u_{R^*}(0),\ x \in B_{R^*}$ and from \eqref{eq soluzione palla}, we have
\begin{equation} \label{u leq u R*}
u(x) \leq \ints_0^{R^*} g'\left(\frac{s}{N}\right) ds,
\end{equation}
for every $x\in\overline{\Omega}$.

Now, we consider the $P$-function given by \eqref{P}. As proved in Lemma 3.2 in \cite{FGK}, $P$ attains its maximum on the boundary of $\Omega$ and thus
\begin{equation*}
P(x) \leq \max_{\pa \Omega} P = \max_{\pa \Omega} \Phi(|\nabla u|) , \quad x\in\overline{\Omega}.
\end{equation*}
Since $\Phi$ is strictly increasing, then we get
\begin{equation}\label{princMaxPerGradU}
\max_{\overline{\Omega}} |\nabla u(x)| = \max_{\pa \Omega} |\nabla u(x)|,
\end{equation}
i.e. $|\nabla u|$ attains its maximum on the boundary of $\Omega$.

Following \cite{GT}, we construct a barrier function for $u$ which will give us an upper bound for $|\nabla u|$ on the boundary of $\Omega$. Let $x_0\in\pa \Omega$ be fixed and let $B_\rho(y(x_0))$ be the ball in the exterior sphere condition. Set
$$\delta(x) = \dist(x,\pa B_\rho(y(x_0))),$$
and let $w(x)=\psi (\delta(x))$ be a function depending only on the distance from $\pa B_\rho(y(x_0))$; we have
\begin{equation} \label{eq:divergenzaDiW}
\diver\left\{ f'(|\nabla w|) \frac{\nabla w}{|\nabla w|} \right\} = \psi''(\delta(x)) f''(\psi'(\delta(x))) + f'(\psi'(\delta(x))) \Delta \delta(x).
\end{equation}
Since
\begin{equation*}
|\Delta \delta(x)| = \frac{N-1}{|x-y|} \leq \frac{N-1}{\rho},
\end{equation*}
from \eqref{eq:divergenzaDiW} we obtain
\begin{equation} \label{eq: quasil dep distance}
\diver\left\{ f'(|\nabla w|) \frac{\nabla w}{|\nabla w|} \right\} +1
\leq  \psi''(\delta(x)) f''(\psi'(\delta(x))) +\frac{N-1}{\rho}
f'(\psi'(\delta(x)))+1.
\end{equation}
By choosing
\begin{equation*}
\psi(t) = \ints_0^t g'\left(\frac{\rho}{N-1} \Big(e^{\frac{N-1}{\rho}(R^*-s)} -1 \Big) \right)  ds,
\end{equation*}
the right hand side of \eqref{eq: quasil dep distance} vanishes and
thus $w$ is a supersolution of \eqref{Euler eq 2}. Notice that $\psi'(t)>0$ for $t>0$ and then $\psi(t)>0$ for $t>0$. Since $x\in\Omega$ implies that $\dist(x,\partial B_\rho(x_0))>0$, we have that $w(x)\geq 0$ for $x\in\overline{\Omega}$. The weak comparison principle yields $u(x) \leq w(x)$ in
$\overline{\Omega}$. Since $x_0 \in \pa \Omega$ is arbitrary, we obtain
\begin{equation*}
|\nabla u(x)| \leq  g'\left(\frac{\rho}{N-1} \Big(e^{\frac{(N-1)R^*}{\rho}} -1 \Big) \right),
\end{equation*}
for any $x\in\partial\Omega$. According to \eqref{princMaxPerGradU} the same
estimate holds in the whole of $\Omega$ and \eqref{bound Du} holds.

Notice that from \eqref{P}
\begin{equation*}
u(x) \leq \frac{N}{2} P(x),\quad x\in\overline{\Omega};
\end{equation*}
since $P$ attains its maximum on the boundary of $\Omega$ and from \eqref{bound Du} we have that
\begin{equation*}
u(x) \leq \frac{N}{2} \Phi(M)
\end{equation*}
which, together with \eqref{u leq u R*}, implies \eqref{bounds u}.
\end{proof}

We denote by $H_{\pa \Omega} (x)$ the mean curvature of $\pa \Omega$ at the point $x\in\pa \Omega$ and set
\begin{equation*}
H_{\pa \Omega}^* =\min_{x\in \pa\Omega} H_{\pa \Omega} (x).
\end{equation*}

In the following Lemma, we give a further bound for $u$ and $|\nabla u|$ in the case that the mean curvature of $\pa \Omega$ attains a positive minimum.

\begin{lemma} \label{lemma: bounds II}
Let $f$ be as in Lemma \ref{lemma:gradients equibounded} and assume that $H_{\pa \Omega}^* > 0$. Then,
\begin{equation}\label{bound u H>0}
u(x) \leq \frac{N}{2} \Phi \bigg( g'\Big( \frac{1}{N H_{\pa \Omega}^*}  \Big) \bigg)
\end{equation}
and
\begin{equation}\label{bound Du H>0}
|\nabla u(x)| \leq g'\bigg( \frac{1}{N H_{\pa \Omega}^*} \bigg)
\end{equation}
for every $x\in \overline{\Omega}$, where $\Phi$ is given by \eqref{Phi} and $g$ is the Fenchel conjugate of $f$.
\end{lemma}
\begin{proof}
Since $|\nabla u|>0$ on $\pa \Omega$ (see Lemma 2.7 in \cite{CMS}), equation \eqref{Euler eq 2} is nondegenerate in a neighborhood of $\pa \Omega$; from standard elliptic regularity theory (see \cite{To} and \cite{GT}), we know that $u\in C^{2,\al}(\overline{\Omega} \setminus \{ x:\ \nabla u \neq 0 \})$ for some $\al\in(0,1)$, and then \eqref{Euler eq 2} can be written pointwise on $\pa \Omega$ as
\begin{equation*}
f''(|u_\nu(x)|) u_{\nu \nu}(x) - (N-1) f'(u_\nu(x)) H_{\pa \Omega}(x) = -1;
\end{equation*}
here, $\nu$ denotes the exterior unit normal to $\pa \Omega$, $u_\nu=\nabla u \cdot \nu$ and $u_{\nu \nu}=(D^2 u) \nu \cdot \nu $. From Lemma 3.3 in \cite{FGK}, we know that
\begin{equation*}
N f'(|\nabla u(x)|) H_{\pa \Omega}(x) \leq 1,
\end{equation*}
for every $x\in\pa\Omega$ and, since $g'$ is nondecreasing, then
\begin{equation*}
|\nabla u(x)| \leq g'\Big(\frac{1}{NH_{\pa \Omega}^*}\Big),
\end{equation*}
for every $x\in\pa\Omega$. Since $\Phi$ is nondecreasing and $P$ (given by \eqref{P}) attains its maximum on $\pa\Omega$, from $u=0$ on $\pa \Omega$ we obtain
\begin{equation} \label{bound P H>0}
P(x) \leq \Phi\bigg( g'\Big(\frac{1}{NH_{\pa \Omega}^*}\Big) \bigg)
\end{equation}
for every $x \in \Omega$. From \eqref{P} and \eqref{bound P H>0} we conclude.
\end{proof}

\noindent Notice that, when $\Omega$ is a ball, \eqref{bound Du H>0} is optimal.

\begin{proposition} \label{prop:approx minimizers}
Let $(f_n)_{n \in \NN}$ be such that:
\begin{enumerate}[(i)]
\item $f_n \in C^1([0,+\infty)) \cap C^3((0,+\infty))$;
\item $f_n$ converges uniformly to $f$ on the compact sets contained in $[0;+\infty)$;
\item $f_n'(0)=0$, the functions $f_n'$ decrease to $f'$ in $[0,+\infty)$ and  $f_n'$ converges uniformly to $f'$ on the compact sets contained in $[0,+\infty)$
\item $f_n''(t)>0$ for $t>0$.
\end{enumerate}
Let $u$ (resp. $u_n$) be the solution of \eqref{J(v)} for $J$ (resp. of \eqref{Jn(v) pb min} for $J_n$).
Then
\begin{enumerate}[(a)]
\item $u_n$ is a minimizing sequence for $J$ and $J_n(u_n) \to J (u)$;
\item $u_n$ and $\nabla u_n$ are uniformly bounded and (up to a subsequence) $(u_n)_{n\in\NN}$ tends to $u$ in the sup norm topology and $u$ satisfies estimates \eqref{bound Du} and \eqref{bounds u} almost everywhere in $\Omega$.
\end{enumerate}
\end{proposition}

\begin{proof}
Since $J_n \to J$ uniformly (a) is standard. Since the sequence $(f_n')_{n\in\NN}$ is decreasing in $n$, then $g_n'$ is increasing in $n$ and converges pointwise to $g'$ (here, we denote by $g$ and $g_n$ the Fenchel conjugates of $f$ and $f_n$, respectively). Thus, $g_n(t) \leq g(t)$ and $g_n'(t) \leq g'(t)$ for every $t \in [0,+\infty)$ and (b) follows by Lemma \ref{lemma:gradients equibounded} and an application of Ascoli-Arzel\à's theorem.
\end{proof}

\section{Viscosity Euler-Lagrange equation}  \label{section 3}
In this section we prove that the solution $u$ of \eqref{J(v)} satisfies an equation of the form \eqref{eq min introd} in the viscosity sense. Firstly, we do it for $f\in C^2((0,+\infty)) \cup C^3((\sigma,+\infty))$ and then we deal with the case that $f$ is not twice differentiable at $s=\sigma$.

Consider a sequence of approximating functions $\{f_n\}_{n\in\NN}$ satisfying $(i)--(iv)$ in Proposition \ref{prop:approx minimizers}. The minimizer $u_n$ for \eqref{Jn(v) pb min} satisfies
\begin{equation*}
-\diver \frac{f_n'(|\nabla u_n|)}{|\nabla u_n|} = 1,
\end{equation*}
in weak sense. Assume for a moment that $u_n$ is regular enough so that we can differentiate, then $u_n$ satisfies
\begin{equation*}
- \frac{|\nabla u_n| f_n''(|\nabla u_n|) - f_n'(|\nabla u_n|)}{|\nabla u_n|^3} \Delta_\infty u_n - \frac{f_n'(|\nabla u_n|)}{|\nabla u_n|} \Delta u_n = 1,
\end{equation*}
where
\begin{equation*}
\Delta_\infty u = \sum_{i,j=1}^N \frac{\pa^2 u}{\pa x_i \pa x_j} \frac{\pa u}{\pa x_i}\frac{\pa u}{\pa x_j}.
\end{equation*}
Since this equation is fully nonlinear and degenerate elliptic, it makes sense to define and study its viscosity solutions (see \cite{CC}).

Let $P\in\RR^N$ and $X\in\mathcal{S}^N$, where $\mathcal{S}^N$ is the space of real-valued $N\times N$ symmetric matrices. Consider the function
\begin{equation}\label{Fn}
F_n(P,X):= \begin{cases}
- \dfrac{|P| f_n''(|P|)-f_n'(|P|)}{|P|^3} P\cdot X P - \dfrac{f_n'(|P|)}{|P|} \tr(X) - 1, & P\neq 0,\\
-1, & P=0.
\end{cases}
\end{equation}
Notice that, if
\begin{equation} \label{fn hp II}
\lim\limits_{s \to 0^+} \dfrac{s f_n''(s) - f_n'(s)}{s^3}=0, \quad \textmd{ and } \quad \lim\limits_{s \to 0^+} \dfrac{f_n'(s)}{s}=0,
\end{equation}
then $F_n$ is continuous. For future use, we shall assume that the sequence $\{f_n\}_{n\in\NN}$ is such that
\begin{equation} \label{fn hp III}
\lim\limits_{n \to +\infty} \frac{s f_n''(s) - f_n'(s)}{s^3}=0\quad \textmd{and } \quad \lim\limits_{n \to +\infty} \frac{f_n'(s)}{s}=0
\end{equation}
uniformly on the compact sets of $[0,\sigma)$; here, thanks to \eqref{fn hp II}, the functions in the limits are understood as continuously extended to $0$ at $s=0$.

\vspace{0.5em}

We shall introduce the definition of viscosity solution of an equation of the form $F(\nabla v,D^2 v)=0$ (see \cite{JLM}).

{\bf Definition.} An upper semicontinuous function $u$ defined in $\Omega$ is a \emph{viscosity
subsolution} of
\begin{equation} \label{eq F_n=0}
F(\nabla v,D^2 v)=0,
\end{equation}
$x\in\Omega$, if, whenever $x_0\in\Omega$ and $\phi\in C^2(\Omega)$ are such that
$u(x_0)=\phi(x_0)$ and $u(x)<\phi(x)$ if $x\neq x_0$, then
\begin{equation*}
F(\nabla \phi(x_0),D^2 \phi(x_0)) \leq 0.
\end{equation*}
A lower semicontinuous function $u$ defined in $\Omega$ is a \emph{viscosity
supersolution} of \eqref{eq F_n=0} if, whenever $x_0\in\Omega$ and $\phi\in C^2(\Omega)$ are such that
$u(x_0)=\phi(x_0)$ and $u(x)>\phi(x)$ if $x\neq x_0$, then
\begin{equation*}
F(\nabla \phi(x_0),D^2 \phi(x_0)) \geq 0.
\end{equation*}
Finally, $u\in C^0(\Omega)$ is a \emph{viscosity solution} of \eqref{eq F_n=0} if it is both a viscosity subsolution and a viscosity
supersolution of \eqref{eq F_n=0}.
\vspace{0.5em}

\begin{lemma}\label{viscosityEqForVApprox}
Let $u_n$ be the minimizer of $J_n$, where $f_n \in C^1([0,+\infty)) \cup C^3((0,+\infty))$ satisfies \eqref{fn hp II} and is such that $f_n''(s)>0$ for $s>0$. Then $u_n$ is a viscosity solution of \eqref{eq F_n=0}, with $F=F_n$ and $F_n$ given by \eqref{Fn}.
\end{lemma}

\begin{proof}
Notice that, since $f_n$ satisfies \eqref{fn hp II}, then $F_n$ is continuous. We present the details for the case of supersolutions. Let $x_0\in\Omega$ and $\phi\in C^2(\Omega)$ be such that $u_n(x_0)=\phi(x_0)$ and $u_n(x)>\phi(x)$ for $x \neq x_0$. Assume that $\nabla \phi(x_0) \neq 0$; we have to show that
\begin{equation*}
- \dfrac{|\nabla \phi (x_0)| f_n''(|\nabla \phi (x_0)|)-f_n'(|\nabla \phi (x_0)|)}{|\nabla \phi (x_0)|^3} \Delta_\infty \phi (x_0) - \dfrac{f_n'(|\nabla \phi (x_0)|)}{|\nabla \phi (x_0)|} \Delta \phi(x_0) - 1 \geq 0.
\end{equation*}
By contradiction, suppose that this is not the case. By continuity, there exists $r> 0$ small enough such that
\begin{equation*}
- \dfrac{|\nabla \phi (x)| f_n''(|\nabla \phi (x)|)-f_n'(|\nabla \phi (x)|)}{|\nabla \phi (x)|^3} \Delta_\infty \phi (x) - \dfrac{f_n'(|\nabla \phi (x)|)}{|\nabla \phi (x)|} \Delta \phi (x) < 1,
\end{equation*}
for any $|x-x_0|<r$. Let $m=\inf\{ u_n(x) - \phi(x):\ |x-x_0|=r\}$ and set $\eta = \phi + \frac{1}{2} m$. Since $m>0$ then $\eta < u_n$ on $\pa B_r(x_0)$, $\eta(x_0)>u_n(x_0)$ and
\begin{equation*}
- \dfrac{|\nabla \eta (x)| f_n''(|\nabla \eta (x)|)-f_n'(|\nabla \eta (x)|)}{|\nabla \eta (x)|^3} \Delta_\infty \eta (x) - \dfrac{f_n'(|\nabla \eta (x)|)}{|\nabla \eta (x)|} \Delta \eta (x) < 1,
\end{equation*}
for any $|x-x_0|<r$. By multiplying by $(\eta-u_n)^+$, integrating in $B_r(x_0)$ and using an integration by parts, we have
\begin{equation} \label{eq weak to visc 1}
\ints_{\{\eta > u_n\}} f_n'(|\nabla \eta|) \frac{\nabla \eta}{|\nabla \eta|} \cdot \nabla (\eta - u_n) dx < \ints_{\{\eta > u_n\}} (\eta - u_n) dx.
\end{equation}
Notice that, since $\eta(x_0)>u_n(x_0)$ and $\eta-u_n$ is continuous, the Lebesgue measure of $\{\eta > u_n\}$ is strictly positive. The function $(\eta-u_n)^+$ extended to zero outside $B_r(x_0)$ can be used as a test function in \eqref{Euler eq 2 fn}:
\begin{equation} \label{eq weak to visc 2}
\ints_{\{\eta > u_n\}} f_n'(|\nabla u_n|) \frac{\nabla u_n}{|\nabla u_n|} \cdot \nabla (\eta - u_n) dx = \ints_{\{\eta > u_n\}} (\eta - u_n) dx.
\end{equation}
Subtracting \eqref{eq weak to visc 2} from \eqref{eq weak to visc 1} we have
\begin{equation} \label{eq weak to visc 3}
\ints_{\{\eta > u_n\}} \left[ f_n'(|\nabla \eta|) \frac{\nabla \eta}{|\nabla \eta|} - f_n'(|\nabla u_n|) \frac{\nabla u_n}{|\nabla u_n|} \right] \cdot \nabla (\eta - u_n) dx < 0.
\end{equation}
Since
\begin{multline}
\left[ f_n'(|\nabla \eta|) \frac{\nabla \eta}{|\nabla \eta|} - f_n'(|\nabla u_n|) \frac{\nabla u_n}{|\nabla u_n|} \right] \cdot \nabla (\eta - u_n) = \\ =
f_n'(|\nabla \eta|)|\nabla \eta| + f_n'(|\nabla u_n|) |\nabla u_n| + \\ - f_n'(|\nabla \eta|) \frac{\nabla \eta}{|\nabla \eta|} \cdot \nabla u_n - f_n'(|\nabla u_n|) \frac{\nabla u_n}{|\nabla u_n|} \cdot \nabla \eta,
\end{multline}
Cauchy-Schwarz inequality yields
\begin{multline*}
\left[ f_n'(|\nabla \eta|) \frac{\nabla \eta}{|\nabla \eta|} - f_n'(|\nabla u_n|) \frac{\nabla u_n}{|\nabla u_n|} \right] \cdot \nabla (\eta - u_n) \geq \\ \geq \Big( f_n'(|\nabla \eta|) - f_n'(|\nabla u_n|) \Big) \Big( |\nabla \eta| - |\nabla u_n| \Big);
\end{multline*}
from the convexity of $f_n$ we obtain
\begin{equation}\label{eq weak to visc 4}
\left[ f_n'(|\nabla \eta|) \frac{\nabla \eta}{|\nabla \eta|} - f_n'(|\nabla u_n|) \frac{\nabla u_n}{|\nabla u_n|} \right] \cdot \nabla (\eta - u_n) \geq 0,
\end{equation}
which gives the desired contradiction on account of \eqref{eq weak to visc 3}.

To complete the proof that $u_n$ is a viscosity supersolution of \eqref{eq F_n=0}, we shall prove that if $\phi$ is a test function touching $u_n$ at $x_0$ from below, then $\nabla \phi (x_0) \neq 0$ (i.e. the set of test functions touching $u_n$ from below with vanishing gradient is the empty set).

By contradition, let us assume that $\phi \in C^2(\Omega)$ is such that $u_n(x_0)=\phi(x_0)$, $u_n(x)>\phi(x)$ for $x \neq x_0$ and $\nabla \phi(x_0) = 0$. Thus, there exists $c>0$ and $r_1>0$ such that $u_n(x)>\phi(x) > \psi(x)$ for $0<|x-x_0|<r_1$, where
\begin{equation*}
\psi(x)=-c |x-x_0|^2 + u_n(x_0).
\end{equation*}
We notice that $\psi$ is of class $C^2$ and satisfies $F_n(\nabla \psi (x), D^2\psi (x)) < 0$ for every $x$ in some ball of radius $r_2$ centered at $x_0$, i.e. there exists $r_2>0$ such that $\psi$ is a strict classical subsolution of $F_n(Dv,D^2v)=0$ in $B_{r_2}(x_0)$. 

Let $r=\min(r_1,r_2)/2$, $m=\inf\{ u_n(x) - \phi(x):\ |x-x_0|=r\}$ and set $\eta = \phi + \frac{1}{2} m$. Notice that $\eta<u_n$ on $\partial B_r(x_0)$, $\eta(x_0)>u_n(x_0)$ and  $F_n(\nabla \eta, D^2 \eta) < 0$ in $B_r(x_0)$. As done in the first part of the proof, we use the function $(\eta-u_n)^+$ extended to zero outside $B_r(x_0)$ as a test function in \eqref{Euler eq 2 fn} and we obtain \eqref{eq weak to visc 3}; then, from \eqref{eq weak to visc 4} we get a contradiction. Thus, the set of test functions touching $u_n$ from below with vanishing gradient is the empty set and hence $u_n$ is a viscosity supersolution of \eqref{eq F_n=0}.

To prove that $u_n$ is a subsolution of \eqref{eq F_n=0}, we first consider a test function $\phi$ touching $u_n$ at $x_0$ from above with $\nabla \phi(x_0) \neq 0$. This case in analogous to the supersolution case. The case $\nabla \phi(x_0)=0$ is simpler than before, since in this case $F_n( \nabla \phi(x_0), D^2\phi(x_0)) \leq 0$ is straightforwardly satisfied.
\end{proof}

\begin{theorem} \label{viscosityEqForV f C2}
Let $u$ be the minimizer of \eqref{J(v)}, with $f$ satisfying \eqref{f hp} and $f\in C^2((0,+\infty))$. Assume that there exists a sequence $\{ f_n \}_{n\in\NN}$ satisfying (i)--(iv) in Proposition \ref{prop:approx minimizers}, \eqref{fn hp II}, \eqref{fn hp III} and such that $f_n''$ converges to $f''$ uniformly on the compact sets contained in $(0,+\infty)$.

Then, $u$ is a viscosity solution of
\begin{equation}\label{eq visc solu u f C2}
\min \left(- \frac{|\nabla u| f''(|\nabla u|) - f'(|\nabla u|)}{|\nabla u|^3} \Delta_\infty u - \frac{f'(|\nabla u|)}{|\nabla u|} \Delta u - 1,\ |\nabla u| - \sigma \right) =0.
\end{equation}
\end{theorem}

\begin{proof}
Let $\{ f_n \}_{n\in\NN}$ be an approximating sequence of the function $f$ satisfying (i)--(iv) in Proposition \ref{prop:approx minimizers}, \eqref{fn hp II}, \eqref{fn hp III} and such that $f_n''$ converges to $f''$ uniformly on the compact sets contained in $(0,+\infty)$. We refer to Theorem \ref{thm exist fn} for the existence of such a sequence in some relevant cases. From Proposition \ref{prop:approx minimizers}, we can assume that $u_n$ converges to $u$ uniformly as $n$ tends to infinity. By using a standard argument from the theory of viscosity solutions (see \cite{CIL} and \cite{Ko}), we shall prove that $u$ is a viscosity supersolution and subsolution of \eqref{eq visc solu u f C2}. The two proofs are not symmetric and we prove firstly that $u$ is a viscosity supersolution and then that it is also a viscosity subsolution.

Assume $\phi$ is a smooth function touching $u$ from below at $\hat{x}\in \Omega$, \ie, $u(\x)=\phi(\x)$ and $u(x)>\phi(x)$ for any $x\neq \x$.
Since $u_n$ is a viscosity solution of \eqref{eq F_n=0} and $u_n$ converges uniformly to $u$, there exist $\{x_n\}_{n\in\NN} \subset \Omega$ such that
  \begin{enumerate}[(i) ]
  \item for any $x\in\Omega$, $u_n(x_n)-\phi(x_n)\leq u_n(x)-\phi(x)$;
  \item $x_n$ tends to $\x$ as $n$ tends to infinity;
  \end{enumerate}
see for instance \cite{JLM} p. 95. Being $u_n$ a viscosity supersolution of \eqref{eq F_n=0}, we can conclude that
\begin{equation*}
    F_n(\nabla \phi(x_n),D^2\phi(x_n)) \geq 0.
\end{equation*}

Let assume that $|\nabla \phi(\x)|<\si$; since $\phi$ is of class $C^2$ and from (ii), there exists $\delta>0$ such that $|\nabla \phi(x_n)| \leq \sigma - \delta$ for $n$ large enough. By taking the limit as $n\to \infty$ and from \eqref{fn hp III} we get a contradiction. Thus, we may exclude that $|\nabla \phi(\x)|<\si$.

Now assume that $|\nabla \phi(\x)| \geq \si$. Since $f_n'$ and $f_n''$ converge uniformly on compact sets to $f'$ and $f''$, respectively, by taking the limit as $n\to \infty$ we get that both
  \begin{equation*}
   - \frac{|\nabla \phi(\x)| f''(|\nabla \phi(\x)|) - f'(|\nabla \phi(\x)|)}{|\nabla \phi(\x)|^3} \Delta_\infty \phi(\x) - \frac{f'(|\nabla \phi(\x)|)}{|\nabla \phi(\x)|} \Delta \phi(\x) - 1 \ge 0,
  \end{equation*}
and
\begin{equation*}
|\nabla \phi(\x)|-\si \geq 0
\end{equation*}
are satisfied. Hence the claim is proven.

Now, we prove that $u$ is a viscosity subsolution of \eqref{eq visc solu u f C2}. Assume $\phi$ is a smooth function such that $u(\x)=\phi(\x)$ and $u(x)<\phi(x)$ for any $x\neq \x$. As claimed at the previous case, there exists a sequence $\{x_n\}_{n\in\NN}$ such that
\begin{enumerate}[(i) ]
 \item $u_n(x_n)-\phi(x_n)\geq u_n(x)-\phi(x)$;
 \item $x_n$ tends to $\x$ as $n$ tends to infinity.
\end{enumerate}
If $|\nabla \phi(\x)|\leq \si$, then obviously
\begin{multline*}
\min \Big(- \frac{|\nabla \phi(\x)| f''(|\nabla \phi(\x)|) - f'(|\nabla \phi(\x)|)}{|\nabla \phi(\x)|^3} \Delta_\infty \phi(\x) - \frac{f'(|\nabla \phi(\x)|)}{|\nabla \phi(\x)|} \Delta \phi(\x) - 1,\\ |\nabla \phi(\x)|-\si \Big) \leq 0
\end{multline*}
holds. In case $|\nabla \phi(\x)|>\si$, then $|\nabla \phi (x_n)| \geq \sigma + \delta$ for some $\delta>0$ and for any $n$ large enough. Since $u_n$ is a viscosity subsolution of \eqref{eq F_n=0}, then we have $F_n(\nabla \phi(x_n),D^2 \phi(x_n)) \leq 0$. Since $f_n$ and its first and second derivatives converges uniformly as $n \to +\infty$, by taking the limit leads to
\begin{equation*}
   - \frac{|\nabla \phi(\x)| f''(|\nabla \phi(\x)|) - f'(|\nabla \phi(\x)|)}{|\nabla \phi(\x)|^3} \Delta_\infty \phi(\x) - \frac{f'(|\nabla \phi(\x)|)}{|\nabla \phi(\x)|} \Delta \phi(\x) - 1 \leq 0,
  \end{equation*}
which completes the proof.
\end{proof}

Now, we assume that $f$ satisfies \eqref{f hp} and
\begin{equation}\label{lim f'' infty}
\lim\limits_{s \to \sigma^+} f''(s) =+\infty.
\end{equation}
Thus, $f \not\in C^2((0,+\infty))$ (i.e. $f$ is not twice differentiable at $s=\sigma$). Since it is not possible to choose $f_n$ such that $f_n''$ converges uniformly to $f''$, we can not proceed as in Theorem \ref{viscosityEqForV f C2}.

Let
\begin{equation}\label{alpha}
a(s) =
\begin{cases}
\dfrac{f'(s)}{s f''(s)},& s>\sigma,\\
0,& 0 \leq s \leq \sigma,
\end{cases}
\end{equation}
and
\begin{equation}\label{gamma}
b(s) =
\begin{cases}
\dfrac{s^2}{f''(s)},& s > \sigma,\\
0, & 0 \leq s \leq \sigma.
\end{cases}
\end{equation}
Notice that $a,b \in C^0([0,+\infty))$. Analogously, we define
\begin{equation}\label{alpha e gamma n}
a_n(s) =\dfrac{f_n'(s)}{s f_n''(s)}\  \textmd{ and } \
b_n(s) = \dfrac{s^2}{f_n''(s)},
\end{equation}
for $s>0$.

\begin{theorem}\label{viscosityEqForV}
Let $u$ be the minimizer of \eqref{J(v)}, with $f$ satisfying \eqref{f hp} and \eqref{lim f'' infty}. Assume that there exists a sequence $\{ f_n \}_{n\in\NN}$ satisfying (i)--(iv) in Proposition \ref{prop:approx minimizers}, \eqref{fn hp II} and \eqref{fn hp III}. Let $a,b,a_n,b_n$ be defined by \eqref{alpha}--\eqref{alpha e gamma n} and assume that $f_n$ is such that $a_n$ and $b_n$ converge uniformly to $a$ and $b$ in the compact sets contained in $(0,+\infty)$ and $(\si,+\infty)$, respectively.

Then $u$ is a viscosity solution of
\begin{equation}  \label{eq:limitForV}
\min \big(-[1 - a(|\nabla u|)]\Delta_\infty u - |\nabla u|^2 a(|\nabla u|)\Delta u - b(|\nabla u|), |\nabla u(x)|-\si \big)=0.
\end{equation}
\end{theorem}

\begin{proof}
The proof splits in two parts. First we prove that $u$ is a viscosity supersolution, then that it is also a subsolution. The earlier is slightly more involved and we deal with it first. Notice that the existence of the sequence $\{ f_n \}_{n\in\NN}$ is proved in Theorem \ref{thm exist fn} for some relevant cases.

\emph{The function }$u$\ \emph{is a viscosity supersolution of
    \eqref{eq:limitForV}}. \\
Assume $\phi$ is a smooth function touching $u$ from below at $\hat{x}\in \Omega$, \ie, $u(\x)=\phi(\x)$ and $u(x)>\phi(x)$ for any $x\neq \x$.
Recall that $u_n$ is a viscosity solution of \eqref{eq F_n=0} and that, from Proposition \ref{prop:approx minimizers}, we can assume that $u_n$ converges uniformly to $u$ as $n$ tends to infinity. Thus, there exist $\{x_n\}_{n\in\NN} \subset \Omega$ such that for any $x\in\Omega$, $u_n(x_n)-\phi(x_n)\leq u_n(x)-\phi(x)$ and $x_n$ tends to $\x$ as $n$ tends to infinity.

Since $u_n$ is a viscosity supersolution of \eqref{eq F_n=0}, we can conclude that
\begin{equation} \label{eq:FmSuccGt0}
    F_n(\nabla \phi(x_n),D^2\phi(x_n)) \geq 0.
\end{equation}

Let assume that $|\nabla \phi(\x)|<\si$. As done in the proof of Theorem \ref{viscosityEqForV f C2}, we get a contradiction and we may exclude that $|\nabla \phi(\x)|<\si$.

Now assume that $|\nabla \phi(\x)| > \si$. Hence, we may assume that $|\nabla \phi(x_n)| > \si$ (at least for $n$ large). By multiplying both sides of \eqref{eq:FmSuccGt0} by
\begin{equation*}
\frac{|\nabla \phi(x_n)|^2}{f_n''(|\nabla \phi(x_n)|)},
\end{equation*}
we have
\begin{equation*}
-\Big[1 - a_n(|\nabla \phi(x_n)|)\Big]\Delta_\infty \phi(x_n) -  |\nabla \phi(x_n)|^2 a_n(|\nabla \phi(x_n)|) \Delta \phi(x_n) - b_n(|\nabla \phi(x_n)|) \geq 0.
\end{equation*}
From the uniform convergence of $a_n$ and $b_n$ and by taking the limit as $n\to \infty$, we get that both
  \begin{equation*}
    -[1- a(|\nabla \phi(\x)|)] \Delta_\infty \phi(\x)-|\nabla \phi(\x)|^2a(|\nabla \phi(\x)|)\Delta\phi(\x)-b(|\nabla \phi(\x)|)\ge 0,
  \end{equation*}
and
\begin{equation*}
|\nabla \phi(\x)|-\si \geq 0
\end{equation*}
are satisfied.

It remains to consider the case $|\nabla \phi(\x)|=\si$. Since we do not have the uniform convergence of $b_n$ to $b$ in a neighborhood of $\si$, we must proceed in a different way. By contradiction, let us assume that $u$ is not a viscosity supersolution of \eqref{eq:limitForV}. For what we have proven in the first part of the proof, there exist $\x \in \Omega$ and a smooth function $\phi$ touching $u$ from below at $\hat{x}\in \Omega$ with $|\nabla \phi(\x)|=\si$ such that
\begin{multline*}
 \min \Big(-[1-a(|\nabla \phi(\x)|)] \Delta_\infty \phi(\x) - |\nabla \phi(\x)|^2 a(|\nabla \phi(\x)|)\Delta \phi(\x) - b(|\nabla \phi(\x)|), \\ |\nabla \phi(\x)|-\si \Big) < 0.
 \end{multline*}
Since $|\nabla \phi(\x)|=\si$, then $a(\si)=b(\si)=0$ and the above inequality yields
\begin{equation}\label{deltainfty <0}
-\Delta_\infty \phi(\x)<0.
\end{equation}
As done before, there exists $\{x_n\}_{n\in\NN} \subset \Omega$ such that for any $x\in\Omega$, $u_n(x_n)-\phi(x_n)\leq u_n(x)-\phi(x)$ and $x_n$ tends to $\x$ as $n$ tends to infinity. Notice that $u_n$ is a viscosity supersolution of \eqref{eq F_n=0} and then
\begin{equation*}
-\frac{f_n''(|\nabla \phi(x_n)|)}{|\nabla \phi(x_n)|^2} \Delta_\infty \phi (x_n) \geq 1 + \frac{f_n'(|\nabla \phi(x_n)|)}{|\nabla \phi(x_n)|} \Delta \phi(x_n) - \frac{f_n'(|\nabla \phi(x_n)|)}{|\nabla \phi(x_n)|^3} \Delta_\infty \phi(x_n).
\end{equation*}
Since $x_n$ converges to $\x$ and $\phi$ is of class $C^2$, then $|\nabla \phi(x_n)|$ converges to $\sigma$ and, from \eqref{deltainfty <0}, $\Delta_\infty \phi(x_n)>0$ for $n$ large enough. The uniform convergence of $f_n'$ to $f'$ yields the following contradiction:
\begin{equation*}
\frac{1}{2} \leq -\frac{f''(|\nabla \phi(x_n)|)}{|\nabla \phi(x_n)|^2} \Delta_\infty \phi (x_n) < 0,
\end{equation*}
for $n$ large enough. Hence the claim is proven.

  \emph{The function }$u$\ \emph{is a viscosity subsolution of
    \eqref{eq:limitForV}}. \\
  Assume $\phi$ is a smooth function such that $u(\x)=\phi(\x)$ and $u(x)<\phi(x)$ for any $x\neq \x$. Thus, there exists a sequence
  $\{x_n\}_{n\in\NN}$ such that $u_n(x_n)-\phi(x_n) \geq u_n(x)-\phi(x)$ and $x_n$ tends to $\x$ as $n$ tends to infinity.

  If $|\nabla \phi(\x)|\leq \si$, then obviously
  \begin{multline*}
 \min \Big(-[1-a(|\nabla \phi(\x)|)] \Delta_\infty \phi(\x) - |\nabla \phi(\x)|^2 a(|\nabla \phi(\x)|)\Delta \phi(\x) - b(|\nabla \phi(\x)|), \\ |\nabla \phi(\x)|-\si \Big) \leq 0
 \end{multline*}
  holds. In case $|\nabla \phi(\x)|>\si$, from the fact that $u_n$ is a
  viscosity subsolution of \eqref{eq F_n=0}, we can conclude
  (carrying out the same algebraic manipulation showed at the previous
  step)
\begin{equation*}
-\Big[1 - a_n(|\nabla \phi(x_n)|)\Big]\Delta_\infty \phi(x_n) -  |\nabla \phi(x_n)|^2 a_n(|\nabla \phi(x_n)|) \Delta \phi(x_n) - b_n(|\nabla \phi(x_n)|) \leq 0.
\end{equation*}
  Taking the limit leads to the desired conclusion.
\end{proof}

\vspace{0.5em}

\begin{rmk}
It is of interest to have an analogue of Theorem \ref{viscosityEqForV} when $f$ satisfies
\begin{equation}\label{lim f'' neq 0 infty}
0 < \lim_{s\to \si^+} f''(s) < +\infty.
\end{equation}
This case can be studied by using an argument analogue to the one used in the proof of Theorem \ref{viscosityEqForV} and under the additional assumption
\begin{equation*}
(\limsup_{n\to \infty} f_n'') (\sigma) \leq \lim_{s\to \si^+} f''(s).
\end{equation*}
We will not write in this paper the details of the proof. We just mention that if $f$ is given by \eqref{f congestion traffic} with $q=2$, then the approximating sequence given in the proof of the following Theorem satisfies this additional assumption.
\end{rmk}

\vspace{0.5em}

Theorems \ref{viscosityEqForV f C2} and \ref{viscosityEqForV} require the existence of a sequence $\{f_n\}_{n\in\NN}$ which satisfies several assumptions. In the following theorem, we construct an explicit example.

\begin{theorem} \label{thm exist fn}
Let $f$ satisfy \eqref{f hp} and let $a$ be given by \eqref{alpha}. Assume that there exists $\tilde{\sigma}>\sigma$ such that $a(s)$ is nondecreasing for $s\in[\sigma,\tilde{\sigma}]$. Then there exists $\{f_n\}_{n\in\NN}$ which satisfies the assumptions required in Theorems \ref{viscosityEqForV f C2} and \ref{viscosityEqForV}.
\end{theorem}

\begin{proof}
We construct an explicit example. A convenient way to construct the sequence $\{f_n\}_{n\in\NN}$ is to modify $f'$ only in the interval $(0,\sigma+\ep)$, with $\ep>0$ small enough. We define
\begin{equation*}
f_\ep'(s)= \begin{cases}
f'(\si+\ep)\Big[ 2 \Big(\dfrac{s}{\si+\ep}\Big)^{p_\ep} - \Big(\dfrac{s}{\si+\ep}\Big)^{q_\ep} \Big], & 0 \leq s \leq \si+\ep, \\
f'(s), & s > \si+\ep,
\end{cases}
\end{equation*}
with
\begin{equation*}
p_\ep = \frac{(\si+\ep)f''(\si+\ep)}{f'(\si+\ep)} \left( 1 + \frac{1}{2} \omega_\ep \: \right),
\end{equation*}
and
\begin{equation*}
q_\ep = \frac{(\si+\ep)f''(\si+\ep)}{f'(\si+\ep)} \left( 1 + \omega_\ep \right),
\end{equation*}
where,
\begin{equation*}
\omega_\ep = \sqrt{2\left[1- \frac{f'''(\si+\ep)f'(\si+\ep)}{f''(\si+\ep)^2} - \frac{f'(\si+\ep)}{(\si+\ep)f''(\si+\ep)}\right]}.
\end{equation*}
Since $a(s)$ is nondecreasing in $[\sigma,\tilde{\sigma}]$, then the same holds for the $\log a(s)$. Thus,
\begin{equation*}
0 \leq \frac{d}{ds} \log a(s) = \frac{f''(s)}{f'(s)} - \frac{1}{s} - \frac{f'''(s)}{f''(s)}.
\end{equation*}
By multiplying by $f'(s)/f''(s)$ we get
\begin{equation*}
\frac{f'(s)}{sf''(s)} + \frac{f'(s)f'''(s)}{f''(s)^2} \leq 1,
\end{equation*}
for $s\in(\si,\tilde{\si})$, which implies that $\omega_\ep$ is well-defined.

Tedious but straightforward computations show that $f_\ep \in C^3(0,+\infty) \cap C^1([0,+\infty))$. Since we modified $f'$ only on a compact set, it is easy to show the uniform convergence of $f_n$ and $f_n'$ to $f$ and $f'$, respectively.

Notice that
\begin{equation*}
f_\ep''(s)= \begin{cases}
\dfrac{f'(\si+\ep)}{\si + \ep}\Big[ 2 p_\ep \Big(\dfrac{s}{\si+\ep}\Big)^{p_\ep-1} - q_\ep\Big(\dfrac{s}{\si+\ep}\Big)^{q_\ep-1} \Big], & 0 \leq s \leq \si+\ep, \\
f''(s), & s > \si+\ep;
\end{cases}
\end{equation*}
since $2p_\ep>q_\ep$ and $p_\ep<q_\ep$, then $f_\ep''>0$.

Notice that we have
\begin{equation} \label{lim a(s) 0}
\lim_{s \to \sigma^+} a(s) = 0.
\end{equation}
Indeed, assume by contradiction that there exists $\al>0$ such that $a(s) \to 1/\al$ as $s \to \sigma^+$. Since $a(s)$ is nondecreasing, then $a(s)\geq 1/\al$ for any $s\in[\si,\tilde{\si}]$, which implies that
\begin{equation*}
\frac{d}{ds} \log f'(s) \leq \frac{d}{ds} \log s^\al, \quad s\in(\si,\tilde{\si}).
\end{equation*}
By integrating both sides of the above inequality from $s$ to $\tilde{\si}$ and after simple manipulations, we obtain that
\begin{equation*}
f'(s) \geq f'(\tilde{\si}) \left(\frac{s}{\tilde{\si}}\right)^\al,
\end{equation*}
for any $s\in (\si,\tilde{\si})$. By taking the limit as $s\to \si^+$, we obtain $f'(\si)>0$, a contradiction. Thus, \eqref{lim a(s) 0} holds.

Since $p_\ep,q_\ep \geq (\si+\ep)f''(\si+\ep)/f'(\si+\ep)$, from \eqref{lim a(s) 0} we have
\begin{equation} \label{lim p e q}
\lim_{\ep\to 0^+} p_\ep = +\infty \ \textmd{ and } \ \lim_{\ep\to 0^+} q_\ep = +\infty.
\end{equation}
Assume that $\ep$ is small enough so that $p_\ep$ and $q_\ep$ are greater than $3$; then
\begin{equation*}
\frac{sf_\ep''(s) - f_\ep'(s)}{s^3} = f'(\si+\ep) \left[ \frac{2(p_\ep-1)}{(\si+\ep)^{p_\ep}} s^{p_\ep-3} - \frac{(q_\ep-1)}{(\si+\ep)^{q_\ep}} s^{q_\ep-3} \right],\quad 0 < s \leq \si+\ep,
\end{equation*}
and
\begin{equation*}
\frac{f_\ep'(s)}{s} = f'(\si+\ep) \left[ \frac{2p_\ep}{(\si+\ep)^{p_\ep}} s^{p_\ep-1} - \frac{q_\ep}{(\si+\ep)^{q_\ep}} s^{q_\ep-1} \right],\quad 0 < s \leq \si+\ep.
\end{equation*}
From \eqref{lim p e q}, we get \eqref{fn hp II} and \eqref{fn hp III}.

We notice that
\begin{equation*}
a_\ep(s)= \frac{f'(\si+\ep)}{(\si+\ep) f''(\si+\ep)} \cdot \frac{2\big(\frac{s}{\si+\ep}\big)^{p_\ep} - \big(\frac{s}{\si+\ep}\big)^{q_\ep}}{2\big(\frac{s}{\si+\ep}\big)^{p_\ep} - \big(\frac{s}{\si+\ep}\big)^{q_\ep} + \omega_\ep \big[{\big(\frac{s}{\si+\ep}\big)^{p_\ep} - \big(\frac{s}{\si+\ep}\big)^{q_\ep}}\big]},
\end{equation*}
for $0\leq s \leq \si+\ep$ and $a_\ep(s)=a(s)$ for $s > \si+\ep$. Thus,
\begin{equation*}
\sup_{s\in\RR} |a_\ep(s)-a(s)| \leq \frac{f'(\si+\ep)}{(\si+\ep) f''(\si+\ep)};
\end{equation*}
by \eqref{lim a(s) 0}, we obtain that $a_\ep$ converges uniformly to $a$.

Since $f_\ep''(s)=f''(s)$ for $s\geq \si+\ep$, it is clear that $b_n$ converges uniformly to $b$ in the compact sets contained in $(\si,+\infty)$.
\end{proof}

\begin{rmk}
We notice when $f$ is given by \eqref{f ikonal} or \eqref{f congestion traffic}, then $a$ satisfies the assumptions of Theorem \ref{thm exist fn}. Indeed, if $f$ is given by \eqref{f ikonal}, then we have
\begin{equation*}
a(s)= \begin{cases}
0, & 0\leq s \leq 1,\\
1- \frac{1}{s^2},& s>1,
\end{cases}
\end{equation*}
and then $a(s)$ is nondecreasing.

When $f$ is given by \eqref{f congestion traffic}, then
\begin{equation*}
a(s)= \begin{cases}
0, & 0\leq s \leq 1,\\
\frac{1}{q-1}\Big(1- \frac{1}{s}\Big),& s>1,
\end{cases}
\end{equation*}
which is a nondecreasing function.

\end{rmk}

\begin{example}
Let $f$ be given by \eqref{f ikonal}. Then $$f'(s) = \sqrt{(s^2-1)^+}$$ and $a$ and $b$ in \eqref{alpha} and \eqref{gamma} read as
\begin{equation*}
a(s)= \begin{cases}
0,& 0 \leq s \leq 1,\\
1- \dfrac{1}{s^2}, & s >1,
\end{cases}
\end{equation*}
and
\begin{equation*}
b(s)= s \sqrt{(s^2-1)^+}.
\end{equation*}
\end{example}

\vspace{1em}

We notice that, working as in the proof of Theorem \ref{viscosityEqForV}, we can prove that $u$ satisfies other equations in the viscosity sense which are of the same form as \eqref{eq min introd}. For instance, let $a^*>0$ be such that $a(s) < a^* $ for any $s \geq 0$; then it can be shown that $u$ satisfies
\begin{equation}  \label{eq viscosity final}
\begin{split}
\min \Big\{-\Big[1 + \frac{1-a^*}{a^*-a(|\nabla u|)} \Big] \Delta_\infty u -  \frac{|\nabla u|^2 a(|\nabla u|)}{a^* - a(|\nabla u|)}\Delta u - & \frac{b(|\nabla u|)}{a^* - a(|\nabla u|)}, \\ & \ \ \ |\nabla u(x)|-\si \Big\}=0,
\end{split}
\end{equation}
in the viscosity sense. If $f$ is given by \eqref{f ikonal}, we can choose $a^*=1$ and \eqref{eq viscosity final} reads as
\begin{equation*}
\min \big(-\Delta_\infty u - |\nabla u|^2 ( |\nabla u|^2 -1)^+ \Delta u - |\nabla u|^3 \sqrt{(|\nabla u|^2-1)^+}, |\nabla u(x)|- 1\big)=0.
\end{equation*}

\vspace{2em}
{\it {\bf Acknowledgments.} The author wishes to thank Rolando Magnanini and Simone Cecchini for the many helpful discussions. The author is grateful
for the careful and thoughtful comments of the referees which led to substantial improvements over a first version of this paper.
}

\end{document}